\documentclass[reqno]{amsart}
\usepackage{amsmath,enumerate,etoolbox,booktabs,float}
\usepackage{txfonts}
\renewcommand{\chi}{\chiup}
\RequirePackage[scaled=0.92]{helvet}%
\usepackage{ifthen}
\usepackage{tikz}
\usetikzlibrary{matrix}
\usetikzlibrary{patterns}
\usetikzlibrary{positioning}
\RequirePackage[paperwidth=155mm,paperheight=235mm,tmargin=20mm,lmargin=19mm,rmargin=19mm,bmargin=16mm,asymmetric]{geometry}%
\newtheorem{theorem}{Theorem}
\newtheorem*{theorem*}{Theorem}
\newtheorem{proposition}{Proposition}
\newtheorem*{proposition*}{Proposition}

\newtheorem*{lemma*}{Lemma}

\newtheorem{corollary}{Corollary}
\newtheorem*{corollary*}{Corollary}
\theoremstyle{definition}

\newtheorem*{question*}{Question}

\newtheorem{remark}{Remark}
\newtheorem*{remark*}{Remark}

\newtheorem*{example*}{Example}
\newtheorem{conjecture}{Conjecture}
\patchcmd{\section}{\scshape}{\bfseries}{}{}
\makeatletter
\renewcommand{\@secnumfont}{\bfseries}
\renewcommand\@biblabel[1]{#1.}
\makeatother
\newcommand{\qbinom}[2]{\genfrac{[}{]}{0pt}{}{#1}{#2}}
\newcommand{\GM}{M}

\newcommand{\NN}{{\mathbf N}}
\newcommand{\ZZ}{{\mathbf Z}}
\newcommand{\height}{{\rm ht}}
\newcommand{\GL}{{\rm GL}}
\newcommand{\SL}{{\rm SL}}
\renewcommand{\L}{\mathcal L}
\newcommand{\K}{R}
\newcommand{\diag}{{\rm diag}}
\newcommand{\wt}{{\rm wt}}
\newcommand{\nesting}{{\rm nest}}
\newcommand{\crossing}{{\rm crossing}}
\newcommand{\permcr}{{\rm cr}}
\newcommand{\wex}{{\rm wex}}
\newcommand{\block}{{\rm block}}
\title{Orthogonal polynomials and Smith normal form}
\author{Alexander~R.~Miller}
\address{Centre \'Emile Borel, Institut Henri Poincar\'e, Paris, France}
\thanks{A.~R.~Miller was supported in part by 
the Fondation Sciences Math\'ematiques de Paris.}
\author{Dennis~Stanton}
\address{School of Mathematics, University of Minnesota, Minneapolis, MN 55455, 
United States}
\email{stanton@math.umn.edu}
\date{March 31, 2017}
\begin{document}
\begin{abstract}
Smith normal form evaluations found by Bessenrodt and Stanley 
for some Hankel matrices of $q$-Catalan numbers are proven in two ways.
One argument generalizes the Bessenrodt--Stanley results for   
the Smith normal form of a certain multivariate matrix that refines one studied by 
Berlekamp, Carlitz, Roselle, and Scoville. 
The second argument, which uses orthogonal polynomials,   
generalizes to a number of other Hankel matrices, 
Toeplitz matrices, and Gram matrices. It gives new results for 
$q$-Catalan numbers, $q$-Motzkin numbers, $q$-Schr\"oder numbers, $q$-Stirling numbers, $q$-matching numbers, 
$q$-factorials, $q$-double factorials, as well as generating functions 
for permutations with eight statistics. 
\end{abstract}
\maketitle
\parindent1em
\vspace{-.2in}
\section{Introduction}\label{Intro}
\noindent 
In \cite{BS} Bessenrodt and Stanley gave a Smith normal form evaluation
for a certain matrix that generalizes one studied by 
Berlekamp \cite{Berlekamp1,Berlekamp2}, and Carlitz, Roselle, and Scoville \cite{CRS}. 
They specialized this result to give 
a Smith normal form result on Hankel matrices of $q$-Catalan numbers.
These evaluations use induction and elementary row and column operations. 
In \S\ref{BS:Remark} we give a short direct combinatorial argument which  
generalizes the results in~\cite{BS}. 
But the main purpose of the present paper is to put 
the Hankel results into the combinatorial framework of orthogonal polynomials. 
This combinatorial theory developed 
over the last 30 years 
immediately implies Bessenrodt and Stanley's two Hankel evaluations 
as well as many new ones, see \S\ref{Section:Examples}.

The main new results in this paper are 
\begin{enumerate}[(1)]
\item Theorem~\ref{Main:Theorem} 
for the Smith normal form of Hankel matrices of moments of 
orthogonal polynomials,
\item Theorem~\ref{Toeplitz:SNF} 
for the Smith normal form of Toeplitz matrices of moments of 
biorthogonal polynomials,
\item Theorem~\ref{Cor:Char} 
for the Smith normal form of a rank matrix of a lattice.
\end{enumerate}

\section{Definitions}
\noindent
Let $A$ be an $m$-by-$n$ matrix with entries in a commutative ring $R$. 
\subsection{}
We say that $A$ has {\it Smith normal form} (or {\it SNF} for short) $D$ over $R$ if
\begin{enumerate}[(a)]
\item $PAQ=D$ for some $P\in\GL(m,\K)$ and $Q\in\GL(n,\K)$,\label{SNF:D}
\item $D$ is a diagonal ($m\times n$) matrix in the sense that $D_{ij}=0$ for $i\neq j$,\label{SNF:diag}
\item $d_{ii}$ is a multiple of $d_{jj}$ whenever $i\geq j$. \label{SNF:d}
\end{enumerate} 
Most of the time $R$ will be a unique factorization domain 
such as $\ZZ[q]$ so that the SNF of $A$ is unique up 
to units if it exists~\cite[Prop.~8.1]{MR}.
Existence is guaranteed for   
$R=\ZZ$ or any other principal ideal domain,  
but not for other types of unique factorization domains.  
For example if $R=\ZZ[q]$, then a diagonal matrix of the form 
$\diag(q+a_1,q+a_2,\ldots, q+a_n)$ 
($a_i\in \ZZ$) admits a Smith normal form if and only if the $a$'s 
are chosen from a set of two consecutive integers~\cite[Prop.~8.9]{MR}.

\subsection{}
If $A$ is square-shaped with Smith normal form $D$,  
then $\det A$ equals $d_1d_2\ldots d_n$ up to a unit factor in $R$. 
For example if $R=\ZZ[q]$, then $\det A=\pm d_1d_2\ldots d_n$. 
Call $D$ a {\it special Smith normal form} ({\it SSNF}) of $A$ over $\K$ 
if in addition to \eqref{SNF:D}--\eqref{SNF:d} it holds that 
\begin{enumerate}[(a$'$)]
\item $PAQ=D$ for some $P\in\SL(m,\K)$ and $Q\in\SL(n,\K)$.\label{SNF:D:p}
\end{enumerate}
\begin{proposition}
$A$ has SSNF over $\K$ $\Leftrightarrow$
$A$ has SNF over $\K$. If $A$ has $n\times n$ SSNF $D$,~then 
\begin{equation}
\det A=\det D=d_1d_2\ldots d_n.
\end{equation}
\end{proposition}
\begin{proof}
If $A$ has Smith normal form $D$ and 
$P,Q$ satisfy \eqref{SNF:D}, then 
scaling the first row of $D$ by $\det P^{-1}\det Q^{-1}$ 
gives SSNF $D'$ of $A$. The other implications are clear.
\end{proof}

\subsection{} 
A number of well-studied determinant evaluations in combinatorics 
can be sharpened into interesting Smith normal 
form evaluations over the rings $\ZZ[q]$ and $\ZZ[x_1,x_2,\ldots,x_n]$. 
But there seems to be no generic explanation why certain matrices 
admit a Smith normal form. Each one uses a different trick. 
Bessenrodt and Stanley gave two recent examples (Corollary~\ref{Catalan:SNF} below). They refine 
\[\det\, (C_{i+j})_{0\leq i,j\leq n} =1\quad \text{and}\quad 
\det\, (C_{i+j+1})_{0\leq i,j\leq n} =1\]
by first replacing the Catalan numbers $C_n=\frac{1}{n+1}\binom{2n}{n}$ 
with the $q$-Catalan numbers $C_n(q)$ below in \eqref{q:Cat:numbers}, 
and then giving Smith normal form evaluations over $\ZZ[q]$ for  
the Hankel matrices $(C_{i+j}(q))$ and $(C_{i+j+1}(q))$. 
The determinants of these two $q$-Hankel matrices are not new. They are 
well known in the combinatorial study of orthogonal polynomials 
and Theorem~\ref{Main:Theorem} tells us that 
Bessenrodt and Stanley's Smith normal form evaluations are 
completely elucidated by the combinatorics of orthogonal polynomials as well.

\section{SNF of Hankel matrices of moments of 
orthogonal polynomials}\label{SNF:Hankel:Section}
\noindent
Take two sequences $b=(b_0,b_1,\ldots)$ and $\lambda=(\lambda_1,\lambda_2,\ldots)$ in 
the commutative ring $\K$. Define $p_0(x),p_1(x),\ldots $ 
in $R[x]$ by 
the classical three-term recurrence relation
\begin{equation}\label{three-term}
p_{n+1}(x)=(x-b_n)p_n(x)-\lambda_np_{n-1}(x),\quad \text{$p_{-1}(x)=0, p_0(x)=1$}.
\end{equation}
The $p_n$'s are orthogonal
in that $\mathcal L(p_n(x)p_m(x))=0$ 
whenever $n\neq m$ for some unique linear functional 
$\mathcal L:\K[x]\to\K$ with $\L(1)=1$.
The moments $\mathcal L(x^n)$ 
are called {\em the moments of $\{p_n(x)\}_{n\geq 0}$} and they are described by Motzkin paths.

\subsection{}
A {\em Motzkin path of length $n$} is a map  
$\omega:\{1,2,\ldots,n+1\}\to\NN$ 
such that
$|\omega'|\leq 1$ 
for $\omega':\{1,2,\ldots,n\}\to\ZZ$ 
defined by $\omega'(i)=\omega(i+1)-\omega(i)$.  
Put 
\begin{equation}\label{weight:formula}
{\rm wt}(\omega)=\prod b_{\omega(i)} \lambda_{\omega(j)}
\end{equation}
over $i$ and $j$ such that $\omega'(i)=0$ and  ${\omega'(j)=-1}$. 
Denote by $\L:\K[x]\to \K$ the linear functional 
whose $n$-th moment $\L(x^n)$ is the weighted generating function
\begin{equation}\label{general:moments}
\mu_n=\L(x^n)=\sum_\omega {\rm wt}(\omega)
\end{equation}
$(n=0,1,\ldots)$ over all Motzkin paths $\omega$ of length $n$ such that ${\omega(1)=\omega(n+1)=0}$. 
Then a sign-reversing involution \cite{V} 
tells us that 
\begin{equation}\label{orthogonal:formula}
\L(p_i(x)p_j(x))=\lambda_1\lambda_2\ldots\lambda_i\delta_{ij}.
\end{equation}
The moments $\mu_n$ of $\L$ are therefore the moments of $\{p_n(x)\}_{n\geq 0}$.

\subsection{} 
Our main theorem is the observation   
that the Hankel matrix $H=(\mu_{i+j})_{0\leq i,j\leq n}$ has 
Smith normal form over 
$\ZZ[b,\lambda]=\ZZ[b_0,b_1,\ldots,\lambda_1,\lambda_2,\ldots]$.

\begin{theorem}\label{Main:Theorem}
$(\mu_{i+j})_{0\leq i,j\leq n}$ has SSNF 
$\diag(1,\lambda_1,\lambda_1\lambda_2,\ldots,\lambda_1\lambda_2\ldots\lambda_n)$ 
over $\ZZ[b,\lambda]$.
\end{theorem}

\begin{proof}
Write $P_{ik}$ for the coefficient of $x^k$ in $p_i(x)$.
Let $P=(P_{ik})_{0\leq i,k\leq n}$.  
Then by~\eqref{orthogonal:formula} 
\begin{equation}\label{diag}
PHP^{t}=\diag(1,\lambda_1,\lambda_1\lambda_2,\ldots,\lambda_1\lambda_2\ldots\lambda_n),\quad H=(\mu_{i+j})_{0\leq i,j\leq n}.
\end{equation}
Since $p_m(x)$ is a polynomial over 
$\ZZ[b,\lambda]$ which is monic of degree $m$, 
$P$ is a matrix over $\ZZ[b,\lambda]$ which is lower triangular with 1's on the diagonal. In other words $P$ is a lower unitriangular matrix over $\ZZ[b,\lambda]$.  
\end{proof}

\subsubsection{}\label{Catalan:no:q} 
For example if $b_n=0$ and $\lambda_n=1$, then by \eqref{general:moments} the $n$-th moment $\mu_n$ equals 
the number of length-$n$ Dyck paths (Motzkin paths where $|\omega'|=1$). 
Hence 
\begin{equation}
\mu_n=\begin{cases} C_{n/2} & \text{if $n$ is even,}\\
0 & \text{if $n$ is odd,}\end{cases}
\end{equation}
where $C_n$ is the $n$-th Catalan number given by $C_{n+1}=\sum_{k=0}^n C_kC_{n-k}$, $C_0=1$. 
In this case  
$\mu_n=\binom{n}{\lfloor n/2\rfloor}-\binom{n}{\lfloor{(n-1)/2}\rfloor}$ 
and Theorem~\ref{Main:Theorem} says that the Hankel matrix 
$(\mu_{i+j})_{0\leq i,j\leq n}$ has 
special Smith normal form $\diag(1,1,\ldots,1)$ over $\ZZ$ so that 
${\det\, (\mu_{i+j})_{0\leq i,j\leq n}=1}$.

\subsubsection{}
We know of only two previous results about the Smith normal form of 
a Hankel matrix of $q$-Catalan numbers over a polynomial ring. They  
are the two mentioned above that 
Bessenrodt--Stanley found~\cite[pp.~81--82]{BS} for the $q$-Catalan numbers
\begin{equation}\label{q:Cat:numbers}
C_{n+1}(q)=\sum_{k=0}^n q^k C_k(q) C_{n-k}(q),\quad C_0(q)=1.
\end{equation}
We record them here in parts \eqref{Catalan:A} and \eqref{Catalan:B} of Corollary~\ref{Catalan:SNF}. 
They are elucidated in \S\ref{q:Catalan:BS:Section}
by Theorem~\ref{Main:Theorem} applied to 
the natural $q$-analogue of our first example from \S\ref{Catalan:no:q}.
\begin{corollary}\label{Catalan:SNF}
\!\begin{enumerate}[\rm(a)]
\item The matrix $(C_{i+j}(q))_{0\leq i,j\leq n}$ has SSNF 
$\diag(q^{\binom{0}{2}},q^{\binom{2}{2}},q^{\binom{4}{2}},\ldots,q^{\binom{2n}{2}})$ over $\ZZ[q]$.\label{Catalan:A}
\item The matrix $(C_{i+j+1}(q))_{0\leq i,j\leq n}$ has SSNF 
$\diag(q^{\binom{1}{2}},q^{\binom{3}{2}},q^{\binom{5}{2}},\ldots,q^{\binom{2n+1}{2}} )$ over $\ZZ[q]$.\label{Catalan:B}
\end{enumerate}
\end{corollary}

\section{Examples}\label{Section:Examples}\noindent
Theorem~\ref{Main:Theorem} also gives new results 
for $q$-Catalan numbers, $q$-Motzkin numbers, $q$-Stirling numbers, $q$-Matching numbers, 
$q$-factorials,  $q$-double factorials, as well as    
more striking generating functions such as 
Simion and Stanton's octabasic 
Laguerre moments which count permutations according to eight different statistics.  
There are many interesting moment sequences 
and this is just a sampling. 

\subsection{$\mathbf q$-Catalan}\label{q:catalan:first} 
If $b_n=0$ and $\lambda_n=q^{n-1}$, then $\mu_n$ counts 
length-$n$ Dyck paths according to area between the path 
and the zig-zag one of height $1$ (Fig.~\ref{Example:q-Catalan}) so that 
\begin{equation}
\mu_n=C_n^*(q)=\begin{cases} C_{n/2}(q) & \text{if $n$ is even,}\\
0 & \text{if $n$ is odd.}\end{cases}
\end{equation}
\begin{corollary}\label{OE:Catalan:SNF}
$(C_{i+j}^*(q))_{0\leq i,j\leq n}$ has SSNF $\diag(1,q^{\binom{1}{2}},q^{\binom{2}{2}},\ldots,q^{\binom{n}{2}})$ over $\ZZ[q]$.\qed
\end{corollary}
\noindent
In \S\ref{q:Catalan:BS:Section} below 
we use a general result (Theorem~\ref{SNF:EO:Theorem}) 
to read off the two Bessenrodt--Stanley results directly from this 
first and most basic example of ours.
\begin{figure}[hbt]\centering
\begin{tikzpicture}[scale=.3]
\draw [very thick] (0,0) -- (1,1) -- (2,0) -- (3,1) -- (4,0) -- (5,1) -- (6,0);
\draw [anchor=south west] (1.1,0.1) node{\small$1$};
\draw [anchor=south west] (3.1,0.1) node{\small$1$};
\draw [anchor=south west] (5.1,0.1) node{\small$1$};
\end{tikzpicture}\ \ 
\begin{tikzpicture}[scale=.3]
\draw [very thick] (0,0) -- (1,1) -- (2,0) -- (3,1) -- (4,2) -- (5,1) -- (6,0);
\draw [anchor=south west] (1.1,0.1) node{\small$1$};
\draw [anchor=south west] (4.1,1.1) node{\small$q$};
\draw [anchor=south west] (5.1,0.1) node{\small$1$};
\draw [dotted] (3,1) -- (4,0) --(5,1);
\end{tikzpicture}\ \ 
\begin{tikzpicture}[scale=.3]
\draw [very thick] (0,0) -- (1,1) -- (2,2) -- (3,1) -- (4,0) -- (5,1) -- (6,0);
\draw [anchor=south west] (2.1,1.1) node{\small$q$};
\draw [anchor=south west] (3.1,0.1) node{\small$1$};
\draw [anchor=south west] (5.1,0.1) node{\small$1$};
\draw [dotted] (1,1) -- (2,0) --(3,1);
\end{tikzpicture}\ \ 
\begin{tikzpicture}[scale=.3]
\draw [very thick] (0,0) -- (1,1) -- (2,2) -- (3,1) -- (4,2) -- (5,1) -- (6,0);
\draw [anchor=south west] (2.1,1.1) node{\small$q$};
\draw [anchor=south west] (4.1,1.1) node{\small$q$};
\draw [anchor=south west] (5.1,0.1) node{\small$1$};
\draw [dotted] (1,1) -- (2,0) --(3,1) -- (4,0) -- (5,1);
\end{tikzpicture}\ \ 
\begin{tikzpicture}[scale=.3]
\draw [very thick] (0,0) -- (1,1) -- (2,2) -- (3,3) -- (4,2) -- (5,1) -- (6,0);
\draw [anchor=south west] (3.1,2.1) node{\small$q^2$};
\draw [anchor=south west] (4.1,1.1) node{\small$q$};
\draw [anchor=south west] (5.1,0.1) node{\small$1$};
\draw [dotted] (1,1) -- (2,0) -- (4,2);
\draw [dotted] (2,2) -- (4,0) -- (5,1);
\end{tikzpicture}%
\caption{$C_3^*(q)=1+2q+q^2+q^3$.}\label{Example:q-Catalan}
\end{figure}
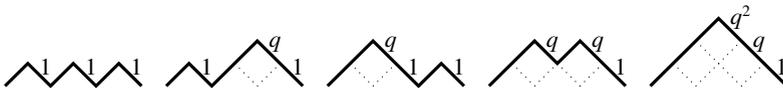

\subsection{$\mathbf q$-Motzkin} 
The $q$-Motzkin number given by
${\rm Motz}_n(q)=\sum_{k= 0}^{n/2}\binom{n}{2k}C_k(q)$ 
is the $n$-th moment $\mu_n$ when $b_n=1$ and $\lambda_n=q^{n-1}$.
\begin{corollary}
$({\rm Motz}_{i+j}(q))_{0\leq i,j\leq n}$ has SSNF $\diag (1,q^{\binom{1}{2}},q^{\binom{2}{2}},\ldots, q^{\binom{n}{2}})$ over $\ZZ[q]$.\qed
\end{corollary}

\subsection{$\mathbf q$-Stirling}\label{Section:Stirling}
The Charlier polynomials $C_n^a(x)$ have moments $\mu_n=\sum_{k=0}^n S(n,k)a^k$ 
where $S(n,k)$ is the Stirling number of the second kind which counts partitions of 
$[n]=\{1,2,\ldots,n\}$ 
into $k$ blocks.  
M\'edicis--Stanton--White \cite{MSW} defined $q$-Charlier polynomials 
\begin{equation}
C_{n+1}^a(x;q)=(x-aq^n-[n]_q)C_n^a(x;q)-a q^{n-1} [n]_qC_{n-1}^a(x;q)
\end{equation}
 and showed that the  
moments are the $q$-analogues $\mu_n=B_n(a,q)$ given by 
$q$-Stirling numbers 
\begin{equation}\label{q:Bell:first}
B_n(a,q)=\sum_{k=0}^n S_q(n,k)a^k,
\end{equation} 
\begin{equation}\label{q:Stirling}
S_q(n,k)=S_q(n-1,k-1)+[k]_q S_q(n-1,k),\quad S_q(0,k)=\delta_{0,k}
\end{equation} 
where $[n]_q=1+q+\ldots+q^{n-1}$. 
The combinatorial interpretation 
of these moments in terms of set partitions 
$\pi$ uses the number of blocks, $\block(\pi)$, and 
another statistic $r\!s(\pi)$. If $\Pi_n$ is the 
set of all set partitions of $[n]$, 
\begin{equation}
\mu_n=B_n(a,q)=\sum_{\pi\in\Pi_n} a^{{\rm{blocks}}(\pi)}q^{r\!s(\pi)}.
\end{equation}
\begin{corollary}
$(B_{i+j}(a,q))_{0\leq i,j\leq n}$ has SSNF $\diag(1,a^1q^{\binom{1}{2}}[1]!_q,
a^2q^{\binom{2}{2}}[2]!_q,\ldots, a^nq^{\binom{n}{2}}[n]!_q)$ over $\ZZ[a,q]$.\qed
\end{corollary}

\subsection{}
\noindent
Kim--Stanton--Zeng \cite{KSZ} defined another 
sequence of $q$-Charlier polynomials
\begin{equation}
C_{n+1}(x,a;q)=(x-a-[n]_q)C_n(x,a;q)-a[n]_qC_{n-1}(x,a;q).
\end{equation}
They showed that the moments are the generating 
functions $\mu_n=\tilde{B}_n(a,q)$ given by 
\begin{equation}
\tilde{B}_n(a,q)=\sum_{\pi\in\Pi_n} a^{{\rm block}(\pi)}q^{\crossing(\pi)}.
\end{equation}
Here $\crossing(\pi)$ is the number of crossings 
in the diagram that has  
$1,2,\ldots, n$ 
written out along a horizontal line 
and an upper arc $i\to j$ for each pair $i<j$ such that 
$j$ is the next largest element in the block containing $i$. 
See Figure~\ref{Example:Pi:Crossings}. 

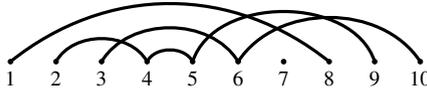
\begin{figure}[hbt]\centering
\begin{tikzpicture}[scale=.3]
\foreach \x in {1,...,10}
          \draw [fill] (2*\x,0) circle [radius=.1]  node[below]{\footnotesize$\x$};
\foreach \x/\y/\a in {3/6/0,2/4/0,4/5/0,5/9/-15,1/8/12,6/10/0} 
\draw [very thick] (2*\x,0) to [out=65-(\y-\x)*2-\a,in=180-(65-(\y-\x)*2-\a),relative] (2*\y,0);
\end{tikzpicture}
\caption{
The partition $\pi=\{\{1,8\},\{2,4,5,9\},\{3,6,10\},\{7\}\}$
drawn above has ${\rm block}(\pi)$ equal to $4$ and $\crossing(\pi)$ equal to $5$.
}\label{Example:Pi:Crossings}
\end{figure}

\begin{corollary}%
$(\tilde{B}_{i+j}(a,q))_{0\leq i,j\leq n}$ has SSNF $\diag(1,a^1[1]!_q,a^2[2]!_q,\ldots, a^n[n]!_q)$ over $\ZZ[a,q]$.\qed
\end{corollary}

\subsection{$\mathbf q$-Matchings} 
Ismail--Stanton--Viennot \cite{ISV} tell us that the polynomials given~by 
\begin{equation}
h_{n+1}(x)=(x-1)h_n(x)-q^{n-1}[n]_qh_{n-1}(x)
\end{equation}
have moments the matching polynomials $\mu_n={\rm Match}_n(q)$ given by  
\begin{equation}
{\rm Match}_n(q)=\sum_m q^{\crossing(m)+2\nesting(m)}=\sum_{k=0}^{n/2} \binom{n}{2k}[1]_q[3]_q\ldots [2k-1]_q.
\end{equation}
The first sum is over all matchings $m$ of $[n]$ 
(partitions of $[n]$ into blocks of size at most 2)
and 
${\rm nest}(m)$ is the number of pairs $\{i,j\},\{k,l\}\in m$ such that $i<k<l<j$.

\begin{corollary}%
$({\rm Match}_{i+j}(q))_{0\leq i,j\leq n}$ has SSNF $\diag(1,q^{\binom{1}{2}}[1]!_q,
q^{\binom{2}{2}}[2]!_q,\ldots, q^{\binom{n}{2}}[n]!_q)$ over $\ZZ[q]$.\qed
\end{corollary}

\subsection{$\mathbf q$-Perfect matchings}\label{q:perfect:matchings}
(Ismail--Stanton--Viennot \cite{ISV}) 
Replacing $x$ by $x+1$ in the last example gives the discrete $q$-Hermite polynomials 
\begin{equation}\label{res:disc:Hermits}
\tilde{h}_{n+1}(x)=x\tilde{h}_n(x)-q^{n-1}[n]_q\tilde{h}_{n-1}(x)
\end{equation}
whose moments $\mu_n=PM_n(q)$ count perfect matchings by crossings and nestings:
\begin{equation}
PM_n(q)=\sum_m q^{\crossing(m)+2\nesting(m)}=\begin{cases}
[1]_q[3]_q\ldots [n-1]_q &\text{if $n$ is even,}\\
0 & \text{if $n$ is odd,}
\end{cases}
\end{equation}
where the sum is over all perfect matchings $m$ of $[n]$ (all blocks of size $2$).
\begin{corollary}%
$(PM_{i+j}(q))_{0\leq i,j\leq n}$ has SSNF $\diag(1,q^{\binom{1}{2}}[1]!_q,q^{\binom{2}{2}}[2]!_q,\ldots, q^{\binom{n}{2}}[n]!_q)$ over $\ZZ[q]$.\qed
\end{corollary}

\subsection{Odd--Even trick}\label{OE:Section} 
In general if the $b_n$'s are all $0$, then 
the polynomials $p_n(x)$ are alternately even and odd 
so that there exist polynomials 
$e_n(x)$ and $o_n(x)$ that satisfy 
\begin{equation}
p_{2n}(x)=e_n(x^2),\quad p_{2n+1}(x)=xo_n(x^2).\label{EO:Polynomials}\end{equation}
The {\it odd-even trick} is the following observation. 
The polynomials $\{e_n(x)\}_{n\geq 0}$ and $\{o_n(x)\}_{n\geq 0}$ are themselves 
orthogonal polynomials and their moments are 
related to the moments $\mu_n$ of the original 
polynomials $\{p_n(x)\}_{n\geq 0}$ in a simple way.

\begin{proposition}[{{\cite[p.~40]{Chihara}}}]\label{Even:Odd}
If $b=(0,0,\ldots)$, then $\mu_{2n+1}=0$ and 
\begin{enumerate}[\rm(i)]
\item\label{EO:Trick:Evens} $\mu_{2n}$ is the $n$-th moment of the sequence $\{e_n(x)\}_{n\geq 0}$ defined by 
\[
e_{n+1}(x)=(x-\lambda_{2n}-\lambda_{2n+1})e_n(x)-\lambda_{2n-1}\lambda_{2n}e_{n-1}(x),
\quad e_{-1}(x)=0,\ e_0(x)=1,
\]
\item\label{EO:Trick:Odds} $\mu_{2n+2}$ is $\lambda_1$ times the $n$-th moment of the sequence $\{o_n(x)\}_{n\geq 0}$ 
defined by
\[
o_{n+1}(x)=(x-\lambda_{2n+1}-\lambda_{2n+2})o_n(x)-\lambda_{2n}\lambda_{2n+1}o_{n-1}(x),
\quad o_{-1}(x)=0,\ o_0(x)=1,
\]
\end{enumerate} 
for $\lambda_0=0$. 
These polynomials $e_n(x)$ and $o_n(x)$ are the unique ones that satisfy \eqref{EO:Polynomials}.
\end{proposition}

\begin{corollary}\label{EO:SNF:Corollary}
If $b=(0,0,\ldots)$ and $\mu_n'',\mu_n'$ are the $n$-th moments 
of the polynomials $\{e_n(x)\}$, $\{o_n(x)\}$ defined by \eqref{EO:Polynomials}, 
then over the ring $\ZZ[\lambda]=\ZZ[\lambda_1,\lambda_2,\ldots]$ 
\begin{enumerate}[\rm(i)]
\item the matrix $(\mu_{i+j}'')_{0\leq i,j\leq n}$ has SSNF $\diag(1,\lambda_1\lambda_2,\lambda_1\lambda_2\lambda_3\lambda_4,\ldots,\lambda_1\lambda_2\ldots\lambda_{2n-1}\lambda_{2n})$,\label{SNF:Even:Polynomials}
\item the matrix $(\mu_{i+j}')_{0\leq i,j\leq n}$ has SSNF $\diag(1,\lambda_2\lambda_3,\lambda_2\lambda_3\lambda_4\lambda_5,\ldots, \lambda_2\lambda_3\ldots \lambda_{2n}\lambda_{2n+1})$.\label{SNF:Odd:Polynomials}
\end{enumerate}
\end{corollary}
\noindent
Corollary~\ref{EO:SNF:Corollary} and Proposition~\ref{Even:Odd} together 
refine the determinant identity (cf.~\cite[Ex.~8.8]{Chihara})
\begin{equation}
\det\, 
(\mu_{i+j})_{0\leq i,j\leq n}
=
\Big[
\det\, 
(\mu_{2i+2j})_{0\leq i,j\leq \lfloor n/2 \rfloor}
\Big]
\Big[
\det\,(\mu_{2i+2j+2})_{0\leq i,j\leq \lfloor (n-1)/2\rfloor}
\Big]
\end{equation}
which holds for $b=(0,0,\ldots)$. 
This is the next result.

\begin{theorem} \label{SNF:EO:Theorem}
Put $s_k=\lambda_1\lambda_2\ldots \lambda_k$ so that $s_0=1$. 
If $b=(0,0,\ldots)$, then over $\ZZ[\lambda]$ 
\begin{enumerate}[\rm(i)]
\item the matrix $(\mu_{i+j})_{0\leq i,j\leq n}$ has SSNF $\diag(s_0,s_1,s_2,\ldots, s_n)$, 
\label{SNF:Full}
\item the matrix $(\mu_{2i+2j})_{0\leq i,j\leq n}$ has SSNF $\diag(s_0,s_2,s_4,\ldots, s_{2n})$,\label{SNF:Even}
\item the matrix $(\mu_{2i+2j+2})_{0\leq i,j\leq n}$ has SSNF $\diag(s_1,s_3,s_5,\ldots,s_{2n+1})$.\label{SNF:Odd}
\end{enumerate}
\end{theorem}
\begin{proof}
\eqref{SNF:Full} restates Theorem~\ref{Main:Theorem}. 
\eqref{SNF:Even} restates Corollary~\ref{EO:SNF:Corollary}\eqref{SNF:Even:Polynomials} 
by Proposition~\ref{Even:Odd}\eqref{EO:Trick:Evens}. 
For \eqref{SNF:Odd} take 
Corollary~\ref{EO:SNF:Corollary}\eqref{SNF:Odd:Polynomials} 
which implies that the matrix $(\lambda_1\mu_{i+j}')_{0\leq i,j\leq n}$ has SSNF 
$\diag(s_1,s_3,\ldots, s_{2n+1})$ and then 
use Proposition~\ref{Even:Odd}\eqref{EO:Trick:Odds} to rewrite $\lambda_1\mu_{i+j}'$ as $\mu_{2i+2j+2}$.
\end{proof}

\subsection{$\mathbf q$-Catalan}\label{q:Catalan:BS:Section} 
Theorem~\ref{SNF:EO:Theorem} applies to our first and most basic example 
in \S\ref{q:catalan:first} and gives the Bessenrodt--Stanley result in 
Corollary~\ref{Catalan:SNF}. 
In the case of the $q$-Chebyshev polynomials from \S\ref{q:catalan:first} where
\begin{equation}
p_{n+1}(x)=xp_n(x)-q^{n-1}p_{n-1}(x)
\end{equation}
Proposition~\ref{Even:Odd} says that 
$C_0(q),C_1(q),C_2(q),\ldots$ is the moment sequence for 
\begin{equation}\label{Cat1}
q_{n+1}(x)=(x-q^{2n}-q^{2n-1}1_{\{n>0\}})q_n(x)-q^{4n-3}q_{n-1}(x)
\end{equation}
and $C_1(q),C_2(q),C_3(q),\ldots$ is the moment sequence for 
\begin{equation}\label{Cat2}
q_{n+1}(x)=(x-q^{2n}(1+q))q_n(x)-q^{4n-1}q_{n-1}(x).
\end{equation}

\subsection{$\mathbf q$-Double factorials} 
Theorem~\ref{SNF:EO:Theorem} also applies to 
the $q$-Hermite example in \S\ref{q:perfect:matchings} and  
gives Corollary~\ref{O:E:qHermit} below. 
In this case Proposition~\ref{Even:Odd} tells us that 
the $q$-double factorials 
${[2n-1]!!_q=[1]_q[3]_q\ldots [2n-1]_q}$ 
are the moments of the polynomials 
where ${b_n=q^{2n-1}[2n]_q+q^{2n}[2n+1]_q}$ and ${\lambda_n=q^{4n-3}[2n-1]_q[2n]_q}$.
\begin{corollary}\label{O:E:qHermit}
$([2i+2j-1]!!_q)_{0\leq i,j\leq n}$ has SSNF $\diag(1,q^{\binom{2}{2}}[2]!_q,q^{\binom{4}{2}}[4]!_q,\ldots, q^{\binom{2n}{2}}[2n]!_q)$ over $\ZZ[q]$.\qed
\end{corollary}

\subsection{$\mathbf q$-Factorials} 
Kasraoui--Stanton--Zeng \cite{KaSZ} defined $q$-Laguerre polynomials
\begin{equation}
L_{n+1}(x;q)=(x-y[n+1]_q-[n]_q)L_n(x;q)-y[n]_q^2 L_{n-1}(x;q)
\end{equation}
and showed that 
$\mu_n=W_n(y,q)$ counts permutations 
with respect to the number of weak excedances and crossings:
\begin{equation}
W_n(y,q)=\sum_{\sigma\in S_n} y^{\wex(\sigma)}q^{\permcr(\sigma)}.
\end{equation}
The number of {\it weak excedances of $\sigma$} is defined by 
\begin{equation}
\wex(\sigma)=\#\{i\in [n] : i\leq \sigma(i)\}
\end{equation}
and the {\it number of crossings of $\sigma$} is defined by 
\begin{equation}
\permcr(\sigma)=\sum_{j=1}^n
\#\{j: j<i\leq \sigma(j)<\sigma(i)\}+\sum_{j=1}^n\#\{j:j>i>\sigma(j)>\sigma(i)\}.
\end{equation}

This may be explained by the following diagram, see Figure~\ref{Example:WEX}. 
With $1$ through $n$ arranged in that order on a horizontal line, 
view $\sigma$ graphically by taking each $i$ and drawing an arc 
$i\to\sigma(i)$ above the line if $\sigma(i)>i$ and 
below the line if $\sigma(i)<i$. Then 
$\wex(\sigma)$ is the number of arcs above the line plus 
the number of isolated points, and $\permcr(\sigma)$ 
is the number of proper crossings plus the number 
of points $1,2,\ldots,n$ at which two different upper arcs meet.

\begin{figure}[hbt]\centering
\begin{tikzpicture}[scale=.3]
\foreach \x in {1,...,16}
          \draw [fill] (2*\x,0) circle [radius=.15]  node[below,outer sep=1.5pt]{\scriptsize$\x$};
\foreach \x/\y/\a in {
1/7/0,7/3/0,3/4/0,4/1/0,
2/5/0,5/2/0,
6/9/0,9/16/0,16/15/0,15/14/0,14/8/9,8/13/0,13/6/0,
10/11/0,11/10/0} 
{\ifthenelse{\x < \y}{
\draw [very thick] 
(2 * \x,0) to [out=(65-abs(\y-\x)*2-\a)*((\y-\x)/(abs(\y-\x)),in=(180-(65-abs(\y-\x)*2-\a))*((\y-\x)/(abs(\y-\x)),relative] (2 * \y,0);}{\draw [very thick] 
(2 * \y,0) to [out=(65-abs(\y-\x)*2-\a)*((\y-\x)/(abs(\y-\x)),in=(180-(65-abs(\y-\x)*2-\a))*((\y-\x)/(abs(\y-\x)),relative] (2 * \x,0);}}
\end{tikzpicture}
\caption{ 
$\sigma=\text{\small $(1,7,3,4)(2,5)(6,9,16,15,14,8,13)(10,11)(12)$}$ 
is drawn above and has $\wex(\sigma)$ equal to $8$ and $\permcr(\sigma)$ equal to $9$.
}\label{Example:WEX}
\end{figure}
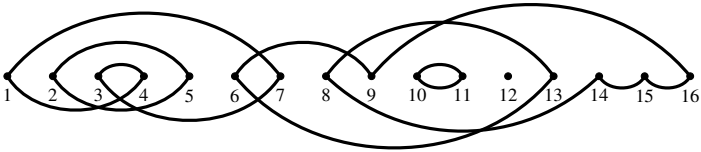

\begin{corollary}%
$(W_{i+j}(y,q))_{0\leq i,j\leq n}$ has SSNF $\diag(1,y^1[1]!_q^2,y^2[2]!_q^2\ldots, y^n[n]!_q^2)$ over $\ZZ[y,q]$. \qed
\end{corollary}

\subsection{} Simion and Stanton's  
{\it octabasic Laguerre polynomials} with  
8 independent $q$'s are defined 
in terms of the 3-term recurrence relation 
\eqref{three-term} by setting 
\begin{equation}
b_n=a[n+1]_{r,s}+b[n]_{t,u},\ \  \lambda_n=ab[n]_{p,q}[n]_{v,w},\ \  [n]_{r,s}=(r^n-s^n)/(r-s).
\end{equation}
The moments are generating functions for permutations 
counted according to eight different statistics which specialize 
to many other combinatorial sets and related statistics~\cite{SS}. In particular Simion--Stanton~\cite{SS1} 
gave specializations whose moments 
are basically $[n]!_q$. If we swap the $a$ and $b$ in their second specialization 
\cite[Eq. 3.3]{SS1}, then we get the polynomials where $\mu_n$ is exactly $[n]!_q$:
\begin{equation}
p_{n+1}(x)=(x-q^n[n+1]_q-q^n[n]_q)p_n(x)-q^{2n-1}[n]_q[n]_q p_{n-1}(x).
\end{equation}

\begin{corollary}%
$([i+j]!_q)_{0\leq i,j\leq n}$ has SSNF $\diag(1,q^{1^2}[1]!_q^2, q^{2^2}[2]!_q^2,\ldots,q^{n^2}[n]!_q^2)$ over $\ZZ[q]$. \qed
\end{corollary}

\section{Bessenrodt--Stanley: general results}\label{BS:Remark}%
\noindent 
Fix a Young diagram $\lambda$. View it in the 
top left part of a square-tiled fourth quadrant. 
Write $(i,j)$ for the square in the $i$-th row and $j$-th column of the tiling. 
Let $d$ be 
the length of the diagonal of $\lambda$, meaning  
$(d,d)\in \lambda$ and $(d+1,d+1)\not\in\lambda$.
Write 
\[\lambda(i,j)=\{(u,v)\in\lambda : i\leq u\text{ and }   j\leq v\}.\] 
Associate to each square $s\in \lambda$ an indeterminate 
$x_s$ and denote by $A_{ij}$ the generating function 
for the skew shapes $\lambda(i,j)\setminus \mu$ so that 
\begin{align}
A_{ij}&\textstyle =
\sum_{\mu\subset\lambda(i,j)}\prod_{s\in\lambda(i,j)\setminus \mu} x_s.\label{A:i:j}
\intertext{The full Bessenrodt--Stanley result \cite[\mbox{Theorem 1}]{BS} 
is Corollary~\ref{SB:Theorem} below for a SSNF of the matrix}
\textstyle A(\lambda)&\textstyle =(A_{ij})_{1\leq i,j\leq d+1}.\label{A:Lambda}
\end{align}
\begin{figure}[hbt]\centering
\begin{tikzpicture}[scale=.3]
\draw [dotted]
(1,13)--(4,13)
(1,12)--(4,12)
(4,14)--(4,13)
(5,14)--(5,13)
(8,14)--(8,13)
(7,14)--(7,13)
(1,8)--(4,8)
(1,7)--(4,7)
;
\node [outer sep=.3pt] at (0,12.5) {$i$};
\node [outer sep=.3pt] at (4.5,15) {$j$};
\node [outer sep=.3pt] at (7.5,15) {$k$};
\node [outer sep=.3pt] at (0,7.5) {$k$};
\draw [dotted] (0+1,15-1) -- (7,8);
\draw 
(1,1) --
(3,1) -- 
(3,3) --
(4,3) -- 
(6,3) --
(6,4) --
(8,4) --
(7+1,6) --
(9+1,6) --
(12,6)--
(12,7) --
(12+1,7) --
(12+1,9) --
(13+1,9) --
(13+1,11) --
(15+1,11) -- 
(15+1,13)--
(15+1+3,13)--
(15+1+3,14)--
(1,14)--cycle;
\draw [ultra thick] 
(4,3) -- (6,3) --
(6,4) --
(8,4) --
(7+1,6) --
(9+1,6) --
(12,6)--
(12,7) --
(12+1,7) --
(12+1,9) --
(13+1,9) --
(13+1,11) --
(15+1,11) -- (15+1,13) -- (4,13) --  cycle;
\draw (7,4) -- (7,13);
\draw (4,8) -- (7+6,8) ;
\fill[pattern=north west lines] (4+.2,3+.2) -- (6-.2,3+.2) -- (6-.2,4+.2) --(7-.2,4+.2) -- (7-.2,8-.2) --(4+.2,8-.2) --(4+.2,3+.2);
\fill[pattern=north west lines] (7+.2,4+.2)--(8-.2,4+.2)--(8-.2,6+.2)--
(12-.2,6+.2)--(12-.2,7+.2)--(13-.2,7+.2)--(13-.2,8-.2)--(7+.2,8-.2)--(7+.2,4+.2);
\fill[pattern=north west lines] 
(7+.2,8+.2)--
(13-.2,8+.2)--
(12+1-.2,9+.2) --
(13+1-.2,9+.2) --
(13+1-.2,11+.2) --
(15+1-.2,11+.2) -- 
(15+1-.2,13-.2)--
(7+.2,13-.2)--(7+.2,8+.2);
\node[fill=white] at (10.5,10.5) {\raisebox{0pt}[.8\height][.5\depth]{\makebox[.8\width]{\footnotesize${U(i,k)}$}}};
\node[fill=white] at (5.5,7) {\raisebox{-.5pt}[.6\height][.0\depth]{\makebox[.7\width]{\footnotesize${L(k,j)}$}}};
\node [fill=white] at (10.5,7) {\raisebox{-.5pt}[.6\height][.0\depth]{\makebox[.7\width]{\footnotesize${\lambda(k,k)}$}}};
\end{tikzpicture}%
\caption{$\lambda(i,j)$ in bold.}%
\label{par:picture}
\end{figure}
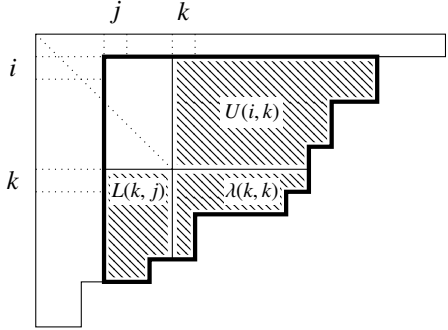
\begin{theorem}\label{SB:LDU}
$A(\lambda)=UDL$ where
\begin{enumerate}[\rm(i)]
\item $U=(U_{ik})_{1\leq i,k\leq d+1}$ is the upper 
triangular matrix 
such that $U_{ik}$ is the generating function 
for skew shapes of $U(i,k)=\{(u,v)\in\lambda\,:\, i\leq u<k\leq v\}$ 
when $i\leq k$,
\item $D=\diag(D_{11},D_{22},\ldots,D_{d+1,d+1})$ 
is the diagonal matrix where 
$D_{kk}=\prod_{s\in\lambda(k,k)}x_s$,
\item $L=(L_{kj})_{1\leq k,j\leq d+1}$ 
is the lower triangular matrix 
such that $L_{kj}$ is the generating function for  
skew shapes of $L(k,j)=\{(u,v)\in\lambda\,:\, j\leq v<k\leq u\}$ 
when $j\leq k$,
\end{enumerate}
so that in particular $L$ and $U$ are lower and upper unitriangular.
\end{theorem}

\begin{proof}
Write $A_{ij}=\sum_{k=1}^{d+1} U_{ik} D_{kk} L_{kj}$ as illustrated by Figure~\ref{par:picture}.
\end{proof}

\begin{corollary}[{{\cite[Thm.~1]{BS}}}]\label{SB:Theorem}
There are upper and lower unitriangular $P$ and $Q$ over 
$\ZZ[x]=\ZZ[x_s:s\in\lambda]$ such that 
\begin{equation}\label{BS:result}
PA(\lambda)Q={\rm diag}(D_{11},D_{22},\ldots,D_{d+1,d+1}),\qquad 
D_{kk}=\textstyle{\prod_{s\in\lambda(k,k)}}x_s.
\end{equation}
In particular, the matrix $A(\lambda)$ has SSNF ${\rm diag}(1,D_{dd},D_{d-1,d-1},\ldots,D_{11})$ over~$\ZZ[x]$.
\begin{proof}
Theorem~\ref{SB:LDU} implies~\eqref{BS:result} 
for $P=U^{-1}$ and $Q=L^{-1}$. But the inverse of an upper (resp. lower)
unitriangular matrix is again upper (resp. lower) unitriangular. 
For the SSNF, 
let $D={\rm diag}(D_{11},D_{22},\ldots,D_{d+1,d+1})$, 
$D'={\rm diag}(1,D_{dd},D_{d-1,d-1},\ldots,D_{11})$, 
and let $X$ be the permutation matrix such that $XDX^{-1}=D'$. 
If $\det{X}=-1$, then 
put $Y={\rm diag}(-1,1,1,\ldots,1)$ so that $\det(YX)=\det(X^{-1}Y)=1$ and 
$YXDX^{-1}Y=D'$.
\end{proof}
\end{corollary}
\noindent 
Bessenrodt--Stanley's two $q$-Catalan results (Corollary~\ref{Catalan:SNF} above)  
are the two cases of Corollary \ref{SB:Theorem} where 
$x_s=q$ and $\lambda=(2n-1,2n-2,\ldots,1),(2n,2n-1,\ldots,1)$.
\noindent
\begin{remark}
Bessenrodt--Stanley gave two more theorems in \cite{BS}. 
Their second theorem is essentially an inclusion-exclusion 
lemma used to recursively construct the $P$ and $Q$ in Corollary~\ref{SB:Theorem}. 
But the nature of the factorization $PAQ=D$ 
implies 
$P=U^{-1}$ and  $Q=L^{-1}$ so we 
can give their $P$ and $Q$ directly 
thanks to unitriangularity. 
The third theorem extends the first to 
some rectangular matrices \cite[Thm. 3]{BS}.  
That theorem can be obtained by specializing to $0$ some variables 
in the first theorem applied to a suitable shape. 
The specialization leads to a more general statement 
(Corollary~\ref{BS:General} below). 
Our direct method works in this case also.
\end{remark}

Let $(a,b)\not\in\lambda$
be a square in the border strip that 
runs from the end of the first column of $\lambda$ 
to the end of the first row of $\lambda$, shown as the shaded region in Figure~\ref{Border:Picture}. 

\begin{figure}[hbt]\centering
\begin{tikzpicture}[scale=.25]
\draw [dotted]
(2,9)--(14,9)
(2,8)--(14,8)
(2,7)--(14,7)
(2,6)--(14,6)
(2,5)--(14,5)
(2,4)--(14,4)
(2,3)--(14,3)
(2,2)--(14,2)
(2,1)--(14,1)

(3,0)--(3,10)
(4,0)--(4,10)
(5,0)--(5,10)
(6,0)--(6,10)
(7,0)--(7,10)
(8,0)--(8,10)
(9,0)--(9,10)
(10,0)--(10,10)
(11,0)--(11,10)
(12,0)--(12,10)
(13,0)--(13,10);

\draw [thick,cap=round] 
(2,10)--(14,10) 
(2,9)--(12,9) 
(2,8)--(11,8) 
(2,7)--(11,7) 
(2,6)--(9,6) 
(2,5)--(6,5) 
(2,4)--(5,4) 
(2,3)--(5,3) 
(2,2)--(4,2)

(2,0)--(2,10) 
(3,2)--(3,10)
(4,2)--(4,10)
(5,3)--(5,10)
(6,5)--(6,10)
(7,6)--(7,10)
(8,6)--(8,10)
(9,6)--(9,10)
(10,7)--(10,10)
(11,7)--(11,10)
(12,9)--(12,10);
\draw [ultra thick,cap=round]
(2,2)--(4,2)--(4,3)--(5,3)--(5,5)
--(6,5)--(6,6)--(9,6)--(9,7)--(11,7)--(11,9)--
(12,9)--(12,10)--(2,10)--cycle
;
\draw [thick,pattern=north west lines,cap=round]
(2,1)--(2,2)--(4,2)--(4,3)--(5,3)--(5,5)
--(6,5)--(6,6)--(9,6)--(9,7)--(11,7)--(11,9)--
(12,9)--(12,10)--
(13,10)--(13,8)--(12,8)--(12,6)--(10,6)--(10,5)--(7,5)
--(7,4)--(6,4)--(6,2)--(5,2)--(5,1)--cycle
;

\end{tikzpicture}
\vspace{-0.1in}
\caption{$\lambda$ in bold.}%
\label{Border:Picture}
\end{figure}
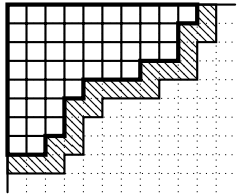

\noindent
Let $\rho$ be the $a\times b$ 
rectangle shape with lower right square $(a,b)$. 
Let $A_{ij}$ be the generating function 
for the skew shapes $\lambda(i,j)\setminus \mu$. 
Write
\begin{equation}
A(\lambda,\rho)=(A_{ij})_{{1\leq i\leq a,\, 1\leq j\leq b}}.
\end{equation} 
Put $c=\min(a,b)$ and define 
\begin{alignat}{2}
U_\rho(i,k)&=\{(u,v)\in\lambda : 
i\leq u<k\leq v+a-b\}\quad &&(1\leq i\leq k\leq a)\\
L_\rho(l,j)&=\{(u,v)\in\lambda : 
j\leq v<l\leq u+b-a\}\quad &&(1\leq j\leq l\leq b)\\
d_i&={\textstyle \prod_{s\in\lambda(a-i+1,b-i+1)}}x_s\quad &&(1\leq i\leq c).
\end{alignat}

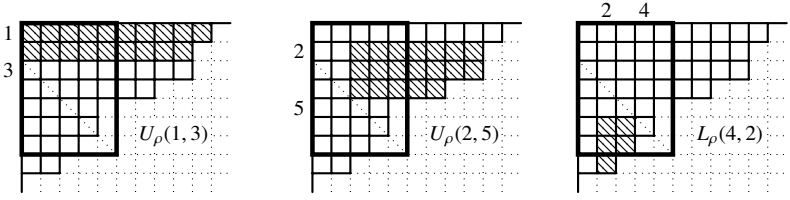
\begin{figure}[hbt]\centering
\hfill
\begin{tikzpicture}[scale=.25]
\draw [dotted]
(2,9)--(13,9)
(2,8)--(13,8)
(2,7)--(13,7)
(2,6)--(13,6)
(2,5)--(13,5)
(2,4)--(13,4)
(2,3)--(13,3)
(2,2)--(13,2)

(3,1)--(3,10)
(4,1)--(4,10)
(5,1)--(5,10)
(6,1)--(6,10)
(7,1)--(7,10)
(8,1)--(8,10)
(9,1)--(9,10)
(10,1)--(10,10)
(11,1)--(11,10)
(12,1)--(12,10);

\draw [dotted] 
(6+1,4-1)--(2,8)
;

\draw [thick,cap=round] 
(2,10)--(13,10) 
(2,9)--(12,9) 
(2,8)--(11,8) 
(2,7)--(11,7) 
(2,6)--(9,6) 
(2,5)--(6,5) 
(2,4)--(6,4) 
(2,3)--(5,3) 
(2,2)--(4,2)

(2,1)--(2,10) 
(3,2)--(3,10)
(4,2)--(4,10)
(5,3)--(5,10)
(6,4)--(6,10)
(7,6)--(7,10)
(8,6)--(8,10)
(9,6)--(9,10)
(10,7)--(10,10)
(11,7)--(11,10)
(12,9)--(12,10);
\draw [line width=0.6mm,cap=round]
(2,3)--(2,10)--(7,10)--(7,3)--cycle;
\node[fill=white] at (10,4) {\raisebox{0pt}[.8\height][.5\depth]{\makebox[.8\width]{\footnotesize$U_\rho(1,3)$}}};
\fill[pattern=north west lines] 
(2,10)--(12,10)--(12,9)--(11,9)--(11,8)--(2,8)--cycle;
\node at (1.3,9.5) {\footnotesize$1$};
\node at (1.3,7.5) {\footnotesize$3$};
\end{tikzpicture}
\hfill
\begin{tikzpicture}[scale=.25]
\draw [dotted]
(2,9)--(13,9)
(2,8)--(13,8)
(2,7)--(13,7)
(2,6)--(13,6)
(2,5)--(13,5)
(2,4)--(13,4)
(2,3)--(13,3)
(2,2)--(13,2)

(3,1)--(3,10)
(4,1)--(4,10)
(5,1)--(5,10)
(6,1)--(6,10)
(7,1)--(7,10)
(8,1)--(8,10)
(9,1)--(9,10)
(10,1)--(10,10)
(11,1)--(11,10)
(12,1)--(12,10);

\draw [dotted] 
(6+1,4-1)--(2,8)
;

\draw [thick,cap=round] 
(2,10)--(13,10) 
(2,9)--(12,9) 
(2,8)--(11,8) 
(2,7)--(11,7) 
(2,6)--(9,6) 
(2,5)--(6,5) 
(2,4)--(6,4) 
(2,3)--(5,3) 
(2,2)--(4,2)

(2,1)--(2,10) 
(3,2)--(3,10)
(4,2)--(4,10)
(5,3)--(5,10)
(6,4)--(6,10)
(7,6)--(7,10)
(8,6)--(8,10)
(9,6)--(9,10)
(10,7)--(10,10)
(11,7)--(11,10)
(12,9)--(12,10);
\draw [line width=0.6mm,cap=round]
(2,3)--(2,10)--(7,10)--(7,3)--cycle;
\node[fill=white] at (10,4) {\raisebox{0pt}[.8\height][.5\depth]{\makebox[.8\width]{\footnotesize$U_\rho(2,5)$}}};
\fill[pattern=north west lines] 
(4,9)--(11,9)--(11,7)--(9,7)--(9,6)--(4,6)--cycle;
\node at (1.3,8.5) {\footnotesize$2$};
\node at (1.3,5.5) {\footnotesize$5$};
\end{tikzpicture}
\hfill
\begin{tikzpicture}[scale=.25]
\draw [dotted]
(2,9)--(13,9)
(2,8)--(13,8)
(2,7)--(13,7)
(2,6)--(13,6)
(2,5)--(13,5)
(2,4)--(13,4)
(2,3)--(13,3)
(2,2)--(13,2)

(3,1)--(3,10)
(4,1)--(4,10)
(5,1)--(5,10)
(6,1)--(6,10)
(7,1)--(7,10)
(8,1)--(8,10)
(9,1)--(9,10)
(10,1)--(10,10)
(11,1)--(11,10)
(12,1)--(12,10);

\draw [dotted] 
(6+1,4-1)--(2,8)
;

\draw [thick,cap=round] 
(2,10)--(13,10) 
(2,9)--(12,9) 
(2,8)--(11,8) 
(2,7)--(11,7) 
(2,6)--(9,6) 
(2,5)--(6,5) 
(2,4)--(6,4) 
(2,3)--(5,3) 
(2,2)--(4,2)

(2,1)--(2,10) 
(3,2)--(3,10)
(4,2)--(4,10)
(5,3)--(5,10)
(6,4)--(6,10)
(7,6)--(7,10)
(8,6)--(8,10)
(9,6)--(9,10)
(10,7)--(10,10)
(11,7)--(11,10)
(12,9)--(12,10);
\draw [line width=0.6mm,cap=round]
(2,3)--(2,10)--(7,10)--(7,3)--cycle;
\node[fill=white] at (10,4) {\raisebox{0pt}[.8\height][.5\depth]{\makebox[.8\width]{\footnotesize$L_\rho(4,2)$}}};
\fill[pattern=north west lines] 
(0+3,0+2)--(0+3,3+2)--(2+3,3+2)--(2+3,1+2)--(1+3,1+2)--(1+3,0+2)--cycle;
\node at (3.5,10.7) {\footnotesize$2$};
\node at (5.5,10.7) {\footnotesize$4$};
\end{tikzpicture}\hfill
\vspace{-0.1in}
\caption{$\rho$ in bold.}%
\label{shift:U:picture}
\end{figure}

\noindent
Theorem~\ref{SB:LDU} is a special case of the following result. 
It is the case 
where $\rho$ is square-shaped of size $(d+1)\times(d+1)$.
\begin{theorem}\label{Gen:SB:LDU} 
$A(\lambda,\rho)=UDL$ where
\begin{enumerate}[\rm(i)]
\item $U=(U_{ik})_{1\leq i,k\leq a}$ is the upper unitriangular matrix given by 
\[U_{ik}=\begin{cases}
\text{the generating function for skew shapes in $U_\rho(i,k)$}
& 
\text{if $i\leq k$,}\\
0 & \text{otherwise,}\end{cases}\]
\item 
$D=(D_{kl})_{1\leq k\leq a,\, 1\leq l\leq b}$ is the matrix given by  
\[D_{kl}=\begin{cases}
d_i & 
\text{if $(k,l)=(a-i+1,b-i+1)$,}\\
0 & \text{otherwise,}
\end{cases}\]
\item 
$L=(L_{lj})_{1\leq l,j\leq b}$ is the lower unitriangular matrix given by 
\[L_{lj}=\begin{cases}
\text{the generating function for skew shapes in $L_\rho(l,j)$}
& 
\text{if $j\leq l$,}\\
0 & \text{otherwise.}\end{cases}\]
\end{enumerate}
\end{theorem}

\begin{proof}
Write $A_{ij}=\sum_{k=1}^a\sum_{l=1}^b U_{ik}D_{kl}L_{lj}$ as in the proof of 
Theorem~\ref{SB:LDU}.
\end{proof}

\begin{corollary}\label{BS:General}
$A(\lambda,\rho)$ has SSNF $\diag_{a\times b}(1,d_{c-1},d_{c-2},\ldots,d_1,0,\ldots,0)$ over $\ZZ[x]$\linebreak 
(the $a\times b$ matrix $D$ with $D_{11},\ldots, D_{cc}$ as given and $0$'s elsewhere.)\qed
\end{corollary}

\begin{proof}
Corollary~\ref{SB:Theorem} handles $a=b$. Assume $a>b$. 
Let $D$ be as in Theorem~\ref{Gen:SB:LDU}. 
Then 
$XD=\diag_{a\times b}(d_1,d_2,\ldots,d_{c-1},1,0,\ldots,0)$ 
for some $X$ that is a permutation matrix with last row 
possibly scaled by $-1$ so that $\det X=1$.
The proof of Theorem~\ref{Gen:SB:LDU} 
provides $P,Q$ such that $\det(P)=\det(Q)=1$ 
and 
\[P\,\diag(d_1,d_2,d_3,\ldots,d_{c-1},1)\,Q=
\diag(1,d_{c-1},d_{c-2},\ldots,d_1).\] 
Consider the block matrix $P'=\diag(P,I_{a-b})$. 
Then $\det P'=1$ and  
\[P'\, \diag_{a\times b}(d_1,d_2,\ldots,d_{c-1},1,0,\ldots,0)\, Q 
=\diag_{a\times b}(1,d_{c-1},d_{c-2},\ldots,d_1,0,\ldots,0).\]
The case $a<b$ follows by transposing matrices. 
\end{proof}

\section{More Gram matrices}\label{Section:Gram}
\noindent
In this section we give two 
analogues of Theorem~\ref{Main:Theorem}: 
one for biorthogonal 
polynomials (Theorem~\ref{Toeplitz:SNF}), 
and the other for finite lattices (Theorem~\ref{Cor:Char}). 

\subsection{Biorthogonal version of Theorem~\ref{Main:Theorem}.}
There is a version of Theorem~\ref{Main:Theorem} for Toeplitz matrices. 
In an integer polynomial ring take two sequences $b=(b_0,b_1,\ldots)$ and $\lambda=(\lambda_1,\lambda_2,\ldots)$ 
subject to $b_i\neq 0$ and 
define $q_0(z),q_1(z),\ldots $ by
\begin{equation}\label{L:three-term}
q_{n+1}(z)=(z-b_n)q_n(z)-z\lambda_nq_{n-1}(z),\quad \text{$q_{-1}(z)=0, q_0(z)=1$.}
\end{equation}
Let $\L$ be the unique linear functional on $\K[z,z^{-1}]$ determined by  
$\L(1)=1$ and 
\begin{equation}\label{L:Orth}
\L(z^mq_n(1/z))=0\quad (0\leq m<n).
\end{equation}
Then 
\begin{equation}\label{B:Orth}
\L(p_m(z)q_n(1/z))=(-1)^n\frac{\lambda_1\lambda_2\ldots \lambda_n}{b_1b_2\ldots b_n}\delta_{mn}
\end{equation}
for $p_0(z)=1$ and
\begin{equation}\label{Partner:Polys}
 p_m(z)=\frac{z^mq_{m+1}(1/z)-z^{m-1}q_m(1/z)}{(-1)^{m+1}b_0b_1\ldots b_m}\quad (m\geq 1)
\end{equation}
so that $p_m(z)$ is a monic polynomial of degree $m$ over $\ZZ[b_0^{\pm 1},\ldots,b_m^{\pm 1},\lambda_1,\ldots, \lambda_m]$.

Kamioka \cite{K} gave a combinatorial approach to these {\it Laurent biorthogonal polynomials} and 
the moments of $\L$ are in terms of {\it Schr\"oder paths}. 
These are the lattice paths $\omega=(\omega_1,\omega_2,\ldots)$ from the origin to the $x$-axis 
that stay at or above the $x$-axis with steps $\omega_i$ chosen from the following types: 
\begin{equation}
\makebox[0pt]{\begin{tabular}{l@{\quad}l@{\qquad}l@{\qquad}l}%
\\\toprule
$\omega_i$ &  & $\overline{\rm wt}(\omega_i)$ & ${\rm wt}(\omega_i)$ \\ 
\midrule
$NE$ &  $(x,k)\to (x+1/2,k+1)$ & 1 & 1 \\
$E$ & $(x,k)\to(x+1,k)$ & $b_k$ & $1/b_k$ \\
$SE$ & $(x,k)\to (x+1/2,k-1)$ & $\lambda_k$ & $\lambda_k/(b_{k-1}b_k)$\\
\bottomrule\\
\end{tabular}}
\end{equation}
Put $f(\omega)=f(\omega_1)f(\omega_2)\dots$ for $f=\wt,\overline{\wt}$ defined above.
Then the {\it $n$-th positive moment} $\L(x^n)$ $(n\geq 0$) 
is the weighted generating function
\begin{align}
\mu_n&=\sum_\omega {\rm wt}(\omega)\label{Laurent:moment:pos}\\
\intertext{over Schr\"oder paths $\omega$ ending at $(0,n)$, and $\L(x^{-n})$ is the {weighted~generating~function}}
\mu_{-n}&=\sum_\omega \overline{\rm wt}(\omega)\label{Laurent:moment:neg}
\end{align}
over Schr\"oder 
paths $\omega$ ending at $(0,n)$ with first step $\omega_1$ horizontal ($E$); see Figure~\ref{Example:Schroder}.

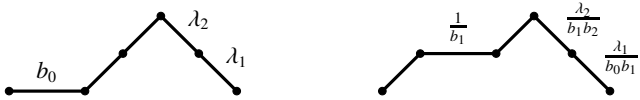
\begin{figure}[hbt]\centering
\hfill
\begin{tikzpicture}[scale=.5]
\draw [very thick] (0,0) --(2,0)--(3,1)--(4,2)--(5,1)--(6,0);
\foreach \c in {(0,0),(2,0),(3,1),(4,2),(5,1),(6,0)}
         \draw [fill] \c circle [radius=.1];
\draw [above] (1,0) node{\small$b_0$};
\draw (4.5+.5,1.5+.4) node{\small$\lambda_2$};
\draw (5.5+.5,0.5+.4) node{\small$\lambda_1$};
\end{tikzpicture}\hfill
\begin{tikzpicture}[scale=.5]
\draw [very thick] (0,0) --(1,1)--(3,1)--(4,2)--(5,1)--(6,0);
\foreach \c in {(0,0),(1,1),(3,1),(4,2),(5,1),(6,0)}
         \draw [fill] \c circle [radius=.1];
\draw [above] (2,1) node{\small$\frac{1}{b_1}$};
\draw (4.5+.8,1.5+.4) node{\small$\frac{\lambda_2}{b_1b_2}$};
\draw (5.5+.8,0.5+.4) node{\small$\frac{\lambda_1}{b_0b_1}$};
\end{tikzpicture}\hfill\hfill
\caption{
$\omega=(E,NE,NE,SE,SE)$ 
has ${\overline{\wt}(\omega)=b_0\lambda_2\lambda_1}$. 
${\omega=(NE,E,NE,SE,SE)}$ 
has ${\wt(\omega)=\frac{1}{b_1}\frac{\lambda_2}{b_1b_2}\frac{\lambda_1}{b_0b_1}}$.
}\label{Example:Schroder}
\end{figure}

\begin{theorem}\label{Toeplitz:SNF}
$(\mu_{i-j})_{0\leq i,j\leq n}$ has SSNF $\diag(1,-\frac{\lambda_1}{b_1},\frac{\lambda_1\lambda_2}{b_1b_2},\ldots,\pm\frac{\lambda_1\lambda_2\ldots \lambda_n}{b_1b_2\ldots b_n})$ 
over $\ZZ[b,b^{-1},\lambda]$.
\end{theorem}
\begin{proof}
Write $P_{ik}$ for the coefficient of $z^k$ in $p_i(z)$, 
and write $Q_{ik}$ for the coefficient of $z^k$ in $q_i(z)$. 
Let $P=(P_{ik})_{0\leq i,k\leq n}$ and $Q=(Q_{ik})_{0\leq i,k\leq n}$.
Then by~\eqref{B:Orth} 
\begin{equation}\label{PHQ}
PTQ^{t}=\diag\left(1,-\frac{\lambda_1}{b_1},\frac{\lambda_1\lambda_2}{b_1b_2},\ldots,(-1)^n\frac{\lambda_1\lambda_2\ldots\lambda_n}{b_1b_2\ldots b_n}\right),\quad T=(\mu_{i-j})_{0\leq i,j\leq n}.
\end{equation}
Since the polynomials 
$p_m(z)$ and $q_m(z)$ are monic of 
degree $m$ over $\ZZ[b,b^{-1},\lambda]$, the matrices  
$P$ and $Q$ are lower unitriangular over $\ZZ[b,b^{-1},\lambda]$. 
\end{proof}

Theorem~\ref{Toeplitz:SNF} gives an interesting 
Schr\"oder analogue of Corollary~\ref{Catalan:SNF}.  
Let $R_n$ be the number of Schr\"oder paths ending at $(0,n)$ so that in terms of 
Catalan numbers $C_n$ then $R_n$ equals $\sum_{k=0}^n \binom{n+k}{n-k}C_k$. 
Put $R_n(q)=\sum_{k=0}^n \binom{n+k}{n-k}C_k(q)$ and consider the Hankel-like matrix
\begin{equation}\label{S:T}%
\left(R_{i+j-1_{\{j>i\}}}(q)\right)_{0\leq i,j\leq n}=
{\footnotesize{\begin{pmatrix}
R_0(q) & R_0(q) & R_1(q) & R_2(q) &\ldots & R_{n-1}(q)\\
R_1(q) & R_0(q) & R_0(q) &             & \ddots  &\vdots\\
R_2(q) & R_1(q) &        & \ddots      &       & R_2(q)\\
R_3(q) &        & \ddots       &       &R_0(q) &R_1(q) \\
\vdots &\ddots  &        & R_1(q)& R_0(q) & R_0(q)\\
R_{n}(q)&\ldots& R_3(q)& R_2(q)& R_1(q)& R_0(q)
\end{pmatrix}}}.
\end{equation}

\begin{corollary}
The matrix in \eqref{S:T} has SSNF $\diag(1,-q^{\binom{1}{2}},q^{\binom{2}{2}},\ldots,
\pm q^{\binom{n}{2}})$ over~$\ZZ[q]$.
\end{corollary}

\begin{proof}
Put $b_n=1$ and $\lambda_n=q^{n-1}$ to get from      
\eqref{Laurent:moment:pos}--\eqref{Laurent:moment:neg} that    
\begin{equation}
\mu_n=\begin{cases} R_n(q) & \text{if $n\geq 0$,}\\
R_{-n-1}(q) & \text{if $n<0$,}
\end{cases}
\end{equation}
and then apply Theorem \ref{Toeplitz:SNF}.
\end{proof}

\subsection{Other Gram matrices}
The Hankel matrix $(\mu_{i+j})_{0\leq i,j\leq n}$ in Theorem~\ref{Main:Theorem} 
can be viewed as the Gram matrix 
$G=(\langle x,y\rangle)_{x,y\in\mathcal O}$ where 
$\mathcal O$ is the set of 
generating functions $p_n(x)$ of 
{\em Favard paths of height $n$} with pairing $\langle p_n(x) p_m(x)\rangle=\mathcal L(p_n(x)p_m(x))$. 
There is an analogous statement for the Toeplitz matrix 
$(\mu_{i-j})_{0\leq i,j\leq n}$ in Theorem~\ref{Toeplitz:SNF}. 
Here are three more examples using Gram matrices. 

\subsubsection{First example} 
Let $L$ be a finite ranked lattice with 
an arbitrary fixed ordering $L=(x_1,x_2,\ldots,x_n)$. 
Take a function $f:L\to \K$ 
and put $\langle x,y\rangle=f(x\vee y)$. Define 
\begin{equation}
G=(f(x\vee y))_{x,y\in L}=(f(x_i\vee x_j))_{1\leq i,j\leq n}.
\end{equation}
Write 
\begin{equation}
Z=(\zeta(x_i,x_j))_{1\leq i,j\leq n},\quad \quad 
\zeta(x,y)=\begin{cases} 1 & \text{if $x\leq y$,}\\ 0 & \text{otherwise.}\end{cases}
\end{equation}
Let $\mu(x,y)$ be the M\"obius function given by 
$(\mu(x_i,x_j))_{1\leq i,j\leq n}=Z^{-1}$. 
Then the function 
$g(x)=\sum_{y\geq x}\mu(x,y)f(y)$ satisfies $f(x)=\sum_{y\geq x}g(y)$.
\begin{proposition}\label{G:Factorization:Prop}
\begin{enumerate}[\rm (a)]
\item \label{G:Factorization} $G=Z\, \diag(g(x_1),g(x_2),\ldots, g(x_n))\, Z^t$ and $Z\in {\rm SL}(n,R)$. 
\item 
If the ordering $L=(x_1,x_2,\ldots,x_n)$ 
is chosen 
so that 
$i\leq j$ whenever $x_i\leq x_j$ (resp. $x_j\leq x_i$), then $Z$ is upper (resp. lower) unitriangular.
\end{enumerate}
\end{proposition}
\begin{proof}
$f(x\vee y)=\sum_{z\geq x,y}g(z)=\sum_{z\in L}\zeta(x,z)g(z)\zeta(y,z)$. 
The rest is clear, since $Z$ is conjugate (by a suitable permutation matrix) 
to an upper unitriangular matrix. 
\end{proof}

\begin{corollary}[Lindstr\"om {{\cite{L}}}] $\det G=\prod_{x\in L}g(x)$.\qed
\end{corollary}
\begin{corollary} \label{Poset:SNF:Cor}
If $\pi$ is a permutation such that 
$g(x_{\pi(i)})$ is a multiple of $g(x_{\pi(j)})$ whenever $i\geq j$, then 
the matrix $G$ has SSNF $\diag(g(x_{\pi(1)}),g(x_{\pi(2)}),\ldots, g(x_{\pi(n)}))$ 
over~$R$.\qed
\end{corollary}
\noindent
The next theorem is a direct consequence 
of Proposition~\ref{G:Factorization:Prop} and Corollary~\ref{Poset:SNF:Cor} for 
$\K=\ZZ[q]$ and $f(x)=q^{{\rm rank}(L)-{{\rm rank}(x)}}$.  
In this case $g(x)$ is the {\em characteristic polynomial} 
$\chi([x,1],q)$ of the interval 
$[x,1]=\{y\in L : x\leq y\}$.

\begin{theorem}\label{Cor:Char} 
Let $L$ be a finite ranked lattice with an ordering $L=(x_1,x_2,\ldots,x_n)$. Let 
$G=(q^{{\rm rank}(L)-{\rm rank}(x_i\vee x_j)})_{1\leq i,j\leq n}$. 
Then the following hold.
\begin{enumerate}[\rm(a)]
\item $G =
Z\, \diag(\chi([x_1,1],q),\chi([x_2,1],q),\ldots,\chi([x_n,1],q))\, Z^t$.
\item \label{Char:SNF}
Suppose that $\chi([x,1],q)$
depends only on the rank of $x$. 
Define $l={\rm rank}(L)$, 
 ${L_k=\{x : {\rm rank}(x)=l-k \}}$, and 
$\chi_k(q)=\chi([x,1],q)$ for $x\in L_k$ ($0\leq k\leq l$).
If $\chi_i(q)$ is a multiple of $\chi_j(q)$ whenever 
$i> j$, then the matrix $G$ has SSNF
$\diag(\chi_0(q)I_{|L_0|},\chi_1(q)I_{|L_1|},\ldots, \chi_{l}(q)I_{|L_l|})$ over $\ZZ[q]$. 
\qed
\end{enumerate}
\end{theorem}
\noindent
Take for example the lattice $\Pi_n$  
of set partitions of $[n]$. Here $x\leq y$ in $\Pi_n$ 
if and only if each block in $x$ is a subset of some   
block in $y$. In particular, the bottom element 
of $\Pi_n$ is the partition with $n$ blocks.  
The top element is the partition with only $1$ block.  
Denote by $|x|$ the number of blocks in $x$ so that $|x|=\block(x)$.

\begin{corollary}
Over $\ZZ[q]$ the matrix $(q^{|x\vee y|})_{x,y\in \Pi_n}$ has SSNF  
\begin{equation}
\diag(qI_{S(n,1)},q(q-1)I_{S(n,2)},\ldots,q(q-1)\ldots(q-n+1)I_{S(n,n)})
\end{equation}
where $S(n,k)$ is the Stirling number of the second kind 
given by 
\begin{equation}
q^n=\sum_{k=0}^n q(q-1)\ldots (q-k+1)S(n,k).
\end{equation}
\end{corollary}
\begin{proof}
There are exactly $S(n,k)$ elements $x$ in $\Pi_n$ 
such that $|x|=k$. For each one 
$\chi([x,1],q)=(q-1)\ldots(q-k+1)$. 
Hence by Theorem~\ref{Cor:Char}\eqref{Char:SNF} with $L=\Pi_n$ 
the matrix $(q^{|x\vee y|-1})_{x,y\in\Pi_n}$ 
has SSNF $\diag(I_{S(n,1)},(q-1)I_{S(n,2)},\ldots,(q-1)\ldots(q-n+1)I_{S(n,n)})$. 
Scaling by $q$ gives the result. 
\end{proof}

\subsubsection{Second example}
\noindent 
Let $x\in{\rm NC}_n$ be a noncrossing partition of $[n]$. 
Associate to $x$ the permutation $\sigma(x)\in S_n$ 
that has one cycle $(i_1\, i_2\, \ldots\, i_k)$ for 
each block ${\{i_1<i_2<\ldots<i_k\}}\in x$. 
The partition $\{1,2,\ldots,n\}$ 
corresponds to the long cycle $c=(1\,2\,\ldots\, n)$. 
The {\it dual partition} $x'\in{\rm NC}_n$ 
corresponds to $\sigma(x)^{-1}c$. Define 
\begin{equation}
J_n(q,\delta)=\Big(q^{|x\vee_{\Pi_n}y|}\delta^{|x'\vee_{\Pi_n}y'|}\Big)_{x,y\in {\rm NC}_n}
\end{equation}
and 
\begin{equation}
J_n(q)=J_n(q,1)=\Big(q^{|x\vee_{\Pi_n}y|}\Big)_{x,y\in {\rm NC}_n}.
\end{equation}

Dahab \cite{Dahab} expressed the determinant of $J_n(q,\delta)$ in terms of 
{\it Beraha factors} $f_k(z)$. 
Define polynomials $p_1(z),p_2(z),\ldots$ by the three-term recurrence relation 
\begin{equation}
p_{k+1}(z)=b_k p_k(z)-p_{k-1}(z),\qquad p_{-1}(z)=0,\ p_0(z)=1,
\end{equation}
\begin{equation}
b_k=
\begin{cases} z & \text{if $k$ is even,}\\
1 & \text{if $k$ is odd.}\end{cases}
\end{equation}
Then $f_k(z)$ ($k\geq 1$) is the unique irreducible factor of $p_k(z)$ over $\ZZ[z]$ that  
is a factor of no previous $p_j(z)$ ($j<k$). 
More explicitly,   
$f_k(z)$ ($k\geq 1$) is the minimal polynomial 
of 
$4\cos^2(\frac{\pi}{k+1})$, and is given by  
\begin{equation}
f_k(z)=\prod_{{1\leq j\leq (k+1)/{2}}\atop (j,k+1)=1 } \left(z-4\cos^2\frac{\pi j}{k+1}\right),\qquad (k\geq 1).\end{equation}

Dahab proved that \cite[Thm.~1.8.1]{Dahab}
\begin{equation}
\det J_n(z)=\prod_{k=1}^n f_k(z)^{m_k}
\end{equation}
and \cite[Thm.~2.5.2]{Dahab}
\begin{equation}
\det J_n(q,\delta)=\det J_n(q\delta)
\end{equation}
where 
\begin{equation}
m_k=\#\{\text{Dyck paths of length $2n$ and height $\geq k$}\}.
\end{equation}

We conjecture the following refinement of Dahab's 
determinantal evaluations.

\begin{conjecture}\label{Catalan:Sub:Question}
$J_n(q,\delta)$ has SSNF $\diag(s_1(q,\delta)I_{h_1}, s_2(q,\delta) I_{h_2},\ldots, s_n(q,\delta)I_{h_n})$ over $\ZZ[q,\delta]$,  
where 
\begin{equation}
h_k=\#\{\text{Dyck paths of length $2n$ and height $k$}\}
\end{equation}
and 
\begin{equation}
s_k(q,\delta)=f_1(q\delta)f_2(q\delta)\ldots f_k(q\delta)=\prod_{1\leq j\leq k}\prod_{{1\leq i \leq {(j+1)/{2}}}\atop{(i,j+1)=1}} \left(q\delta-4\cos^2\frac{\pi i}{j+1}\right).
\end{equation}
In particular:
\begin{enumerate}[\rm (a)]
\item $J_n(q)$ has SSNF $\diag(s_1(q)I_{h_1}, s_2(q) I_{h_2},\ldots, s_n(q)I_{h_n})$ over $\ZZ[q]$, where   
\begin{align}
s_k(q)=s_k(q,1)=f_1(q)f_2(q)\ldots f_k(q)=\prod_{1\leq j\leq k}\prod_{{1\leq i \leq {(j+1)/{2}}}\atop{(i,j+1)=1}} \left(q-4\cos^2\frac{\pi i}{j+1}\right).
\end{align}
\item $J_n(q,q)$ has SSNF $\diag(s_1'(q)I_{h_1}, s_2'(q) I_{h_2},\ldots, s_n'(q)I_{h_n})$ over $\ZZ[q]$, where   
\begin{align}
s_k'(q)=s_k(q,q)=f_1(q^2)f_2(q^2)\ldots f_k(q^2)=
q\prod_{{1\leq i \leq j\leq k}\atop{(i,j+1)=1}} \left(q-2\cos\frac{\pi i}{j+1}\right).
\end{align}
\end{enumerate}
\end{conjecture}

\subsubsection{Third example}
Take the noncrossing 
perfect matchings $m$ on $[2n]$ 
and put $\langle m,m'\rangle=q^{c(m,m')}$ 
where $c(m,m')$ is the number of connected components 
in the graph on $[2n]$ whose multiset of edges is $m\cup m'$.  
This is Lickorish's form \cite{Lickorish} and 
the determinant of the Gram matrix 
$M_n(q)=(\langle m,m'\rangle)_{m,m'}$ 
has been studied \cite{DGG,KoSmo,Lickorish}, see~\cite{K1}. 
But a straightforward calculation shows that 
\begin{equation}
M_n(q)=q^{-1}J_n(q,q)
\end{equation}
(for some compatible ordering of the matchings and the non-crossing partitions). 
Therefore Conjecture~\ref{Catalan:Sub:Question} implies the following 
conjecture for $M_n(q)$. 

\begin{conjecture}
$\GM_n$ has SSNF $\diag(s_1(q)I_{h_1}, s_2(q) I_{h_2},\ldots, s_n(q)I_{h_n})$ over $\ZZ[q]$ 
where $h_k$ is the number of Dyck paths of length $2n$ and height $k$, and  
\begin{equation}
s_k(q)=\prod_{{1\leq i \leq j\leq k}\atop{(i,j+1)=1}} \left(q-2\cos\frac{\pi i}{j+1}\right).
\end{equation}
\end{conjecture}

\section{Remarks}\label{V:Section}
\noindent
There are new and interesting results   
for other types of matrices as well. 
Some recent examples are found in~\cite{St1,St2}; 
they again refine some well-known determinantal evaluations.  
But many determinantal evaluations (e.g.~\cite{K1,K2}) 
have not been considered. 
Here for example is a new result we found 
for the Vandermonde matrix. 

\begin{theorem}\label{V:SNF} 
Let 
\begin{equation}
g_i(x)=\sum_{k=0}^i A_{ik}a^kx^k,\quad A_{ik}\in \ZZ[a,q],\ A_{ii}=1.
\end{equation} 
Then over $\ZZ[a,q]$ the matrix $(g_i([j]_q))_{0\leq i,j\leq n}$ has SSNF 
\begin{equation}
{\rm diag}\big(1, a^1q^{\binom{1}{2}}[1]!_q,a^2q^{\binom{2}{2}}[2]!_q, \ldots, a^nq^{\binom{n}{2}}[n]!_q \big).
\end{equation}
In particular:
\begin{enumerate}[\rm (a)]
\item Over $\ZZ[a,q]$ the Vandermonde matrix 
$\big((1+a[j]_q)^i\big)_{0\leq i,j\leq n}$ has SSNF 
\begin{equation}
{\rm diag}\big(1, a^1q^{\binom{1}{2}}[1]!_q,a^2q^{\binom{2}{2}}[2]!_q, \ldots, a^nq^{\binom{n}{2}}[n]!_q \big).
\end{equation}
\item Over $\ZZ[q]$ the Vandermonde matrix 
$\big([j+1]_q^i\big)_{0\leq i,j\leq n}$ has SSNF 
\begin{equation}
{\rm diag}\big(1, q^{\binom{2}{2}}[1]!_q,q^{\binom{3}{2}}[2]!_q, \ldots, q^{\binom{n+1}{2}}[n]!_q \big).
\end{equation}
\end{enumerate}
\end{theorem}

Theorem~\ref{V:SNF} is a special case of the 
following generalization 
of Theorem~\ref{Main:Theorem}. 

\begin{theorem}\label{Generalized:Main:Theorem} 
Maintain the notation of {{\S\ref{SNF:Hankel:Section}}} 
so that $\L$ is the linear functional 
for the polynomials $p_k(x)$ defined 
by the three-term recurrence relation 
\[p_{n+1}(x)=(x-b_n)p_n(x)-\lambda_np_{n-1}(x),\quad p_{-1}(x)=0, p_0(x)=1.\]
Let $Y_0(x),Y_1(x),\ldots,Y_n(x),Z_0(x),Z_1(x),\ldots,Z_n(x)$ 
be monic polynomials over $\ZZ[b,\lambda]$ 
such that $Y_k(x)$ and $Z_k(x)$ have degree $k$ 
for $0\leq k\leq n$. Then the matrix 
\begin{equation}
\Big(\L(Y_i(x)Z_j(x))\Big)_{0\leq i,j\leq n}
\end{equation}
has SSNF 
\begin{equation}
{\rm diag}(1,\lambda_1,\lambda_1\lambda_2,\ldots, \lambda_1\lambda_2\ldots \lambda_n)
\end{equation}
over $\ZZ[b,\lambda]$.\qed
\end{theorem}

Theorem~\ref{Main:Theorem} is the special case 
of Theorem~\ref{Generalized:Main:Theorem}
where $Y_k(x)=Z_k(x)=x^k$ for all $k$. 
Theorem~\ref{V:SNF} is the case 
where the polynomials $p_k(x)$ 
are the $q$-Charlier polynomials $C_k^a(x;q)$ 
from \S\ref{Section:Stirling}
and 
\begin{align}
Z_j(x)&=\sum_{u=0}^j \qbinom{j}{u}_q p_u(x),\\
Y_i(x)&=\sum_{t=0}^i\sum_{k=t}^i A_{ik}S_q(k,t)a^{k-t}p_t(x),
\end{align}
where $\qbinom{j}{u}_q=[j]_q[j-1]_q\ldots[j-u+1]_q/[u]!_q$.


\end{document}